\documentclass[aop,preprint]{imsart}
\setattribute{journal}{name}{}

\RequirePackage[OT1]{fontenc}
\RequirePackage{amsthm,amsmath,amsfonts,amssymb,mathrsfs}
\RequirePackage[numbers]{natbib}
\RequirePackage[colorlinks,citecolor=blue,urlcolor=blue]{hyperref}
\usepackage{tikz}
\usetikzlibrary{arrows,calc}
\usepackage{bbm}
\usepackage{enumerate}

\usepackage{mathtools}
\mathtoolsset{showonlyrefs}

\usepackage{environ}
\NewEnviron{eq}{%
\begin{equation}\begin{split}
  \BODY
\end{split}\end{equation}
}

\startlocaldefs

\def\sss{\scriptscriptstyle}

\newcommand{\ind}[1]{\ensuremath{\mathbbm{1}{\left\{#1\right\}}}}

\newcommand{\E}{\ensuremath{\mathbbm{E}}}
\newcommand{\R}{\ensuremath{\mathbbm{R}}}
\newcommand{\N}{\ensuremath{\mathbbm{N}}}

\newcommand{\dif}{\mathrm{d}}

\newcommand{\cN}{\mathcal{N}}

\newcommand{\cW}{\mathcal{W}}

\newcommand{\eqd}{\stackrel{\sss d}{=}}

\newcommand{\sB}{\mathscr{B}}

\newcommand\blfootnote[1]{%
  \begingroup
  \renewcommand\thefootnote{}\footnote{#1}%
  \addtocounter{footnote}{-1}%
  \endgroup
}

\newtheorem{theorem}{Theorem}
\newtheorem*{claim*}{Claim}

\newtheorem{lemma}[theorem]{Lemma}
\newtheorem{proposition}[theorem]{Proposition}
\newtheorem{corollary}[theorem]{Corollary}

\newtheorem{remark}{Remark}

\numberwithin{fact}{section}
\newtheorem{defn}{Definition}

\let\plainqed\qedsymbol

\numberwithin{equation}{section}
\numberwithin{theorem}{section}

\endlocaldefs

\begin{document}

\begin{frontmatter}
\title{A correction to Kallenberg's theorem for jointly exchangeable random measures}
\runtitle{Local finiteness of jointly exchangeable random measures}

\begin{aug}
\author{\fnms{Christian} \snm{Borgs}\thanksref{m1}\ead[label=e1]{Christian.Borgs@microsoft.com}},
\author{\fnms{Jennifer} \snm{T. Chayes}\thanksref{m1}\ead[label=e2]{jchayes@microsoft.com}}
\author{\fnms{Souvik} \snm{Dhara}\thanksref{m1,m2}\ead[label=e3]{sdhara@mit.edu}}
\and
\author{\fnms{Subhabrata} \snm{Sen}\thanksref{m3}
\ead[label=e4]{subhabratasen@fas.harvard.edu}
}
\runauthor{Borgs, Chayes, Dhara, Sen}

\affiliation{Microsoft Research\thanksmark{m1}, Massachusetts Institute of Technology\thanksmark{m2} and Harvard University \thanksmark{m3}}

\address{Microsoft Research,\\
One Memorial Drive,\\
Cambridge MA 02142.\\
\printead{e1}\\
\phantom{E-mail:\ }\printead*{e2}
}

\address{Department of Mathematics\\
Massachusetts Institute of Technology, \\
77, Massachusetts Avenue, Building 2, \\
Cambridge MA 02139. \\
\printead{e3}
}

\address{Department of Statistics\\
Harvard University, \\
One Oxford Street, Suite 400\\
Cambridge, MA 02138. \\
\printead{e4}
}
\end{aug}
\blfootnote{This article will appear as the supplementary material to \cite{BCDS18}.}
\begin{abstract}
Kallenberg (2005) provided a necessary and sufficient condition for the local finiteness  of a jointly exchangeable random measure on $\R_+^2$.
Here we note an additional condition that was missing in Kallenberg's theorem, but was implicitly used in the proof. 
We also provide a counter-example when the additional condition does not hold. 
\end{abstract}
\end{frontmatter}

\section{Characterization of local finiteness for jointly exchangeable measures on $\mathbb{R}^2_+$}
\citet[Theorem 9.24]{K2005} established a representation for all jointly exchangeable random measures on $\R_+^2$. This representation theorem has been the bedrock of  recent developments in the study of sparse graph limits \cite{VR15, VR16, BCCV17} and non-parametric Bayesian inference of network data \cite{CF17}. In this context, it is natural to restrict one's attention to locally finite random measures. For the convenience of the reader, we include a definition of locally finite random measures.
\cite[Definition 9.1.I]{DV08}.
\begin{defn}[Locally-finite random measures] \normalfont
A random measure on $\mathbb{R}_+^2$ is called locally finite if with probability 1, $\xi(B)<\infty$ for all bounded measurable $B \subset \mathbb{R}_+^2$.
\end{defn}

\noindent
 \citet[Proposition 9.25]{K2005} also states a characterization for local finiteness of any such exchangeable random measure. On closer inspection, it turns out that the local finiteness characterization has an extra implicit condition. To redress this issue, we include the complete statement of the characterization below (see Theorem \ref{thm:local_finiteness}) . Further, to make our treatment self-contained, we also include a complete proof in this section. We emphasize that the proof is almost the same as Kallenberg's original argument, once we add the additional condition.
We begin by recalling the notion of a jointly exchangeable measure.  Let $\sB(\R_+)$ denote the set of all Borel subsets of $\mathbb{R}_+$.

\begin{defn}
\normalfont
A measure $\xi$ on $\mathbb{R}^2_+$ is jointly exchangeable if $\xi(\phi^{-1}(A)\times \phi^{-1}(B))\eqd\xi(A\times B)$ for every measure preserving map $\phi:\R_+\to\R_+$, and for all $A,B\in \sB(\R_+)$.
\end{defn}

\begin{theorem}[\cite{K2005}]\label{thm:kallenberg}
A random measure $\xi$ on $\R_+^2$ is jointly exchangeable iff a.s.
\begin{eq} \label{eq:kallenberg-general}
\xi =& \sum_{i,j} f(\alpha, \vartheta_i, \vartheta_j, \zeta_{\{ i,j\}} ) \delta_{\tau_i, \tau_j} + \beta \lambda_{\mathrm{D}} + \gamma \lambda^2 \\
&+ \sum_{j,k} \Big( g(\alpha, \vartheta_j, \chi_{jk} ) \delta_{\tau_j, \sigma_{jk}} + g'(\alpha, \vartheta_j, \chi_{jk} ) \delta_{ \sigma_{jk}, \tau_j} \Big) \\
&+ \sum_{j} \Big( h(\alpha, \vartheta_j) ( \delta_{\tau_j} \otimes \lambda) + h'( \alpha, \vartheta_j) (\lambda \otimes \delta_{\tau_j} )  \Big) \\
&+ \sum_{k} \Big( l(\alpha, \eta_k) \delta_{\rho_k, \rho_k'} + l'(\alpha, \eta_k) \delta_{\rho_k', \rho_k} \Big),
\end{eq}
for some measurable functions $f \geq 0$ on $\R_+^4$, $g, g' \geq 0$ on $\R_+^3$, and $h,h',l,l' \geq 0$ on $\R_+^2$, a collection of iid uniform random variables $\{ \zeta_{\{i , j \} } : i \leq j \}$, some independent, unit rate Poisson processes $\{ (\tau_j , \vartheta_j) \}_{j\geq 1}$ and $\{ (\sigma_{ij} , \chi_{ij} ) \}_{j\geq 1}$ on $\R_+^2$ for each $i \geq 1$, and $\{(\rho_j, \rho_j', \eta_j) \}$ on $\R_3^+$, and an independent set of random variables $\alpha, \beta, \gamma \geq 0$. The latter can then be chosen to be non-random iff $\xi$ is extreme.
\end{theorem}

\noindent
Our interest centers around random adjacency measures, which are purely atomic almost surely. The next corollary, obtained as an immediate consequence of  Theorem~\ref{thm:kallenberg}, provides a representation theorem for all jointly exchangeable random atomic measures on $\R_+^2$.
\begin{corollary}\label{cor:atomic_representation}
A purely atomic random measure $\xi$ on $\R_+^2$ is jointly exchangeable iff a.s.
\begin{align}
\xi =& \sum_{i,j} f(\alpha, \vartheta_i, \vartheta_j, \zeta_{\{ i,j\}} ) \delta_{\tau_i, \tau_j} \nonumber\\
&+ \sum_{j,k} \Big( g(\alpha, \vartheta_j, \chi_{jk} ) \delta_{\tau_j, \sigma_{jk}} + g'(\alpha, \vartheta_j, \chi_{jk} ) \delta_{ \sigma_{jk}, \tau_j} \Big) \nonumber\\
&+ \sum_{k} \Big( l(\alpha, \eta_k) \delta_{\rho_k, \rho_k'} + l'(\alpha, \eta_k) \delta_{\rho_k', \rho_k} \Big),\label{eq:atomic_rep}
\end{align}
where  the functions $f, g, g', l, l'$, and the stochastic components are the same as {\rm Theorem~\ref{thm:kallenberg}}. Further, $\alpha \geq0$ may be chosen to be non-random iff $\xi$ is extreme.
\end{corollary}
\begin{proof}
$\xi$ is almost surely atomic, and thus all components with Lebesgue contributions vanish. This immediately leads to the representation of interest.
\end{proof}

\noindent

To study local finiteness of these random measures, it suffices to establish this characterization in the extreme case, when $\alpha$ is a constant, and thus for convenience of notation, we will suppress the dependence on $\alpha$ in the subsequent discussion. For any function $\phi$, we denote $\hat{\phi} = \phi \wedge 1$, and define
\begin{align*}
f_1(\cdot) = \int_{0}^{\infty} \int_{0}^{1} \hat{f}( \cdot, y,z) \,  \dif z\, \dif y  \,\,\,\,\,
f_2(\cdot) = \int_{0}^{\infty} \int_0^1 \hat{f} (y, \cdot, z) \, \dif y \,\dif z.
\end{align*}
Further, we define
\begin{align*}
g_1(\cdot) = \int_0^{\infty} \hat{g}(\cdot, y) \,\dif y,  \,\,\,\,\,g'_1(\cdot) = \int_0^{\infty}\hat{g}'(\cdot, y) \, \dif y
\end{align*}
For conciseness of notation, for any measurable function $\phi: \R_+ \to \R$, we set $\lambda \phi = \int_0^{\infty} \phi(y) \, \dif y$. To avoid confusion, for $B \in \sB(\R_+)$, we denote the Lebesgue measure as $\lambda\{ B \}$. Similarly, for any point process $\eta= \{ x_j : j \geq 1\}\subset \R_+$, we set $\eta\phi = \sum_{j \geq 1} \phi(x_j)$.  Further, for $\phi : \R_+^2 \to \R$, we define $\eta^2 \phi = \sum_{i,j} \phi(x_i , x_j)$.

\begin{theorem}
\label{thm:local_finiteness}
For a fixed $\alpha$, the random measure \eqref{eq:kallenberg-general} is a.s. locally finite iff the following conditions are satisfied:
\begin{itemize}
\item[(i)] $\lambda ( \hat{l} + \hat{l}' + \hat{h} + \hat{h}' )  < \infty$.
\item[(ii)] $\lambda \{ g_1 = \infty\} = \lambda \{ g'_1 = \infty\} =0$.
\item[(iii)] $\lambda( \hat{g}_1 + \hat{g}'_1) < \infty$.
\item[(iv)] $\lambda\{ f_i = \infty \} =0$ and $ \lambda \{ f_i >1 \} < \infty$ for $i =1,2$.
\item[(v)]  $ \int_0^{\infty} \int_0^{\infty} \int_0^1 \hat{f}(x,y,z) \,\ind{ f_1 (x) \vee f_2 (y)\leq 1 }\, \dif z\, \dif y\, \dif x  < \infty$.
\item[(vi)] $\int_0^{\infty} \int_0^1 \hat{f}(x,x,z) \,\dif z \, \dif x < \infty$.
\end{itemize}
\end{theorem}

\noindent
This theorem is a modification of  \citet[Proposition 9.25]{K2005}. The main difference is that our theorem contains the extra condition (ii), which was missing in \citet[Proposition 9.25]{K2005}, but was used implicitly in the proof.
We provide a complete proof of Theorem \ref{thm:local_finiteness} in the rest of this section.
To establish Theorem \ref{thm:local_finiteness}, we need a preliminary lemma about almost sure convergence of Poisson integrals.
\begin{lemma}[{{\cite[Theorem A3.5]{K2005}}}]\label{lemma:integral_convergence}
Let $\eta$ be a unit rate Poisson process on $\mathbb{R}_+$. Then for any measurable function $f:\mathbb{R}_+ \to \mathbb{R}_+$ and $f \geq 0$, we have $\eta f < \infty$ a.s. iff $\lambda \hat{f} < \infty$. Further, let  $h : \mathbb{R}_+^2 \to \mathbb{R}_+$ be measurable with $h \geq 0$. Setting $h_i = \lambda_j(\hat{h})$ for $j \neq i$, we have, $\eta^2 h < \infty$ a.s. iff
\begin{enumerate}
\item[(i)] $\lambda\{ h_1 = \infty \} = \lambda\{ h_2 = \infty \} = 0$,
\item[(ii)] $\lambda\{ h_1 >1 \} < \infty$, $\lambda\{h_2 >1 \} < \infty$,
\item[(iii)]  $\int_0^{\infty} \int_0^{\infty} \hat{h}(x,y) \ind{( h_1(x) \vee h_2(y) \leq 1 )} \dif y\, \dif x < \infty$
\item[(iv)] $ \int_0^{\infty} \hat{h}(x,x)  \dif x < \infty$
\end{enumerate}
\end{lemma}
\noindent
We refer the interested reader to \cite[Proposition A.2]{BCCL18} for an independent proof in a slightly different setup.
Given Lemma \ref{lemma:integral_convergence}, the proof is relatively straightforward. We include a proof here for the sake of completeness.
We will also use the following elementary lemma about convergence of a random series of independent non-negative random variables, which is a consequence of the Kolmogorov three series theorem.

\begin{lemma}\label{lem:summability}
Let $Z_1, Z_2, \cdots$ be an independent sequence of non-negative random variables. Then $\sum_j Z_j < \infty$ almost surely if and only if $\sum_j \E[1 \wedge Z_j ] < \infty$.
\end{lemma}

\begin{proof}[Proof of Theorem \ref{thm:local_finiteness}]
The measure $\xi$ is jointly exchangeable -- thus for establishing local finiteness, it suffices to restrict the measure to $[0,1]^2$, without loss of generality. Then we can write
\begin{eq}
&\xi([0,1]^2) \\
&= \sum_{i,j} f(\vartheta_i, \vartheta_j , \zeta_{\{ i,j \}} ) \ind{\tau_i \leq 1, \tau_j \leq 1} + \sqrt{2}\beta + \gamma \nonumber\\
 &+ \Big(  \sum_{j,k} g(\vartheta_j, \chi_{jk}) \ind{\tau_j \leq 1, \sigma_{jk} \leq 1} + \sum_{j,k} g'(\vartheta_j , \chi_{jk}) \ind{\sigma_{jk} \leq 1, \tau_j \leq 1} \Big) \nonumber\\
&+ \sum_j h(\vartheta_j) \ind{\tau_j \leq 1} + \sum_j h'( \vartheta_j) \ind{\tau_j \leq 1} \nonumber\\
&+ \sum_{k} \Big( l(\eta_k) \ind{\rho_k \vee \rho_k' \leq 1} + l'(\eta_k) \ind{\rho_k' \vee \rho_k \leq 1} \Big). \label{eq:restriction}
\end{eq}
\noindent
The terms $\sqrt{2} \beta$ and $\gamma$ are trivially finite for all random variables $\beta, \gamma \geq 0$. Introduce the point processes
\begin{align}
\tilde{\vartheta} = \sum_{ j} \delta_{\vartheta_j} \ind{\tau_j \leq 1}\,\,\,\,\,\, \tilde{\eta} = \sum_{k} \delta_{\eta_k} \ind{\rho_k \vee \rho_k' \leq 1}.  \nonumber
\end{align}
We note that these are unit point Poisson processes on $\mathbb{R}_+$, and thus the last four terms in \eqref{eq:restriction} are finite if and only if $ \tilde{\vartheta} h + \tilde{\vartheta} h' + \tilde{\eta} l + \tilde{\eta} l' <\infty $ almost surely. An application of Lemma~\ref{lemma:integral_convergence} immediately characterizes this convergence, and yields condition $(i)$ in Theorem~\ref{thm:local_finiteness}.

Next, we establish that conditions $(ii)$ and $(iii)$ in Theorem \ref{thm:local_finiteness} are sufficient to guarantee almost sure finiteness of the $g, g'$ terms in \eqref{eq:restriction}.
To this end, condition on the Poisson process $\{(\tau_j, \vartheta_j)\}_{j \geq 1}$.
Define
\begin{align}
\tilde{\chi}_j = \sum_{k} \delta_{\chi_{jk}} \ind{ \sigma_{jk} \leq 1 }. \nonumber
\end{align}
Conditionally on the Poisson point process $\{(\tau_j, \vartheta_j)\}_{j \geq 1}$, $(\tilde{\chi}_j)_{j\geq 1}$ forms a collection of independent  unit Poisson point process on $\mathbb{R}_+$.
Note that Condition (ii) of Theorem~\ref{thm:local_finiteness} implies $\sum_{k} g(\vartheta_j, \chi_{jk}) \ind{ \tau_j \leq 1, \sigma_{jk} \leq 1 } < \infty$ almost surely for all $j \geq 1$.
Thus, for all $j\geq 1$, the random variables $\sum_{k} g(\vartheta_{j}, \chi_{jk} ) \ind{ \tau_j \leq 1, \sigma_{jk} \leq 1}$ are independent $\mathbb{R}_+$ valued random variables.
To characterize the convergence of this random series, we condition on $\{(\vartheta_j, \tau_j)\}_{ j \geq 1}$ and apply Lemma \ref{lem:summability}. This yields that the series is finite almost surely iff
\begin{align}\label{eq:g-cond-finite}
\sum_{j} \E\Big[ 1 \wedge \sum_{k} g(\vartheta_j, \chi_{j,k}) \ind{  \sigma_{jk} \leq 1} \Big| \vartheta_j, \tau_j \Big]  {\ind{\tau_j \leq 1}}< \infty.
\end{align}
Let us denote $\psi(x) = 1 - e^{-x}$. Using $\frac{1\wedge x}{2} \leq \psi(x) < 1 \wedge x$ for $x >0$, it is easy to see that \eqref{eq:g-cond-finite} is equivalent to
\begin{align}
&\sum_{j \geq 1} \E\Big[ \psi\Big( \tilde{\chi}_j g(\vartheta_j , \cdot)  \Big) | \vartheta_j , \tau_j  \Big] {\ind{\tau_j \leq 1}} \\
& \hspace{.5cm}= \sum_{j \geq 1} \psi\Big( \int_0^{\infty} \psi( g(\vartheta_j, y )) \,\dif y  \Big) \ind{\tau_j \leq 1} < \infty \nonumber
\end{align}
where the last equality follows from {\cite[Lemma A3.6]{K2005}}. Note that this condition is satisfied once (ii) is given, and ensures the almost sure convergence of the sum, conditioned on the process $\{ (\vartheta_j, \tau_j)\}_{ j \geq 1}$.
Finally, we ``uncondition" on the point process $\{(\vartheta_j, \tau_j)\}_{  j \geq 1}$, and note that  given condition (ii), the almost sure convergence of the $g$ term in \eqref{eq:restriction} is equivalent to
\begin{align}
{\tilde{\vartheta}\bigg(\psi\Big( \int_0^{\infty} \psi( g( \cdot , y )) \,\dif y   \Big)\bigg) < \infty }\nonumber
\end{align}
almost surely.
To characterize the convergence of this sum, we again apply Lemma \ref{lemma:integral_convergence}, and note that this finiteness is equivalent to
\begin{align}
\int_0^{\infty} \psi\Big( \int_0^{\infty} \psi \circ g(x, y )  \dif y  \Big) \dif x < \infty.  \nonumber
\end{align}
Finally, using $\frac{1\wedge x}{2} \leq \psi(x) \leq 1\wedge x$, it is not too hard to see that the condition above is equivalent to (iii). The argument for the $g'$ is exactly same, and is thus omitted. This establishes the sufficiency of (ii) and (iii) for the almost sure finiteness of the relevant terms in \eqref{eq:restriction}. It remains to establish the necessity of these conditions. Indeed, consider the function
\begin{align} \label{eq:counter-example}
g(x,y) = \begin{cases}
1 &\textrm{if}\, x \in [0,1] , y\in [0,1] \cup [2,3] \cup \cdots \\
0 &\textrm{o.w.}
\end{cases}
\end{align}
In this case, condition (ii) is violated, and it is easy to see that the corresponding $g$ term in \eqref{eq:restriction} is infinite with positive probability. Thus (ii) is indeed necessary. Given (ii),  the rest of the proof above is necessary and sufficient, which establishes the necessity of (iii) as well.

Finally, we need to establish necessary and sufficient conditions for the almost sure finiteness of $f$ term in \eqref{eq:restriction}. To this end, first condition on the point process $\{ (\vartheta_j, \tau_j )\}{ j \geq 1 }$, and {using Lemma~\ref{lem:summability}}, the finiteness in this case is equivalent to
\begin{align}
&\sum_{i,j} \E[ 1 \wedge f(\vartheta_i, \vartheta_j , \zeta_{\{i,j\}} ) | \vartheta_i , \tau_i ] \ind{\tau_i \vee \tau_j \leq 1} \\
& \hspace{.5cm} = \sum_{i,j} \hat{f}_3( \vartheta_i, \vartheta_j ) \ind{\tau_i  \leq 1, \tau_j \leq 1} < \infty, \nonumber
\end{align}
where we define $\hat{f}_3 (x,y) = \int_0^1 \hat{f} (x,y,z) \dif z$. Finally, we ``uncondition" on the point process $\{ (\vartheta_j ,\tau_j) : j \geq 1\}$, and note that the convergence of the $f$ term in \eqref{eq:restriction} is equivalent to
\begin{eq}
\sum_{i,j}  \hat{f}_3(\vartheta_i, \vartheta_j) \ind{\tau_i \leq 1 ,\tau_j \leq 1}  =  \tilde{\vartheta}^2 \hat{f}_3  < \infty  \nonumber
\end{eq}
almost surely. The rest of the proof follows by a direct application of Lemma \ref{lemma:integral_convergence}, { and noting that $0\leq \hat{f}_3\leq 1$.}
\end{proof}

\begin{remark}\normalfont \eqref{eq:counter-example} clearly  constructs a counter-example for the local finiteness without condition (ii) in Theorem~\ref{thm:local_finiteness}.
\end{remark}

\section{Characterization of random adjacency measures}
%
The next result characterizes all random adjacency measures on $\R^2_+$. To keep the discussion self-contained, we recall the concepts under consideration. In the subsequent discussion, $\cN(\R_+^2)$  denotes the set of locally finite counting measures on $\mathbb{R}_+^2$, equipped with the vague topology.

\begin{defn}[Random adjacency measure] \label{defn:adj_measure}
\normalfont
An adjacency measure
is a
measure  $\xi\in\cN(\R_+^2)$
such that $\xi(A\times B) = \xi(B\times A)$ for all $A,B\in \sB(\R_+)$.
A random adjacency measure is a $\cN(\R_+^2)$ valued random variable that is almost surely an adjacency measure. It is called \emph{exchangeable} if $\xi(\phi^{-1}(A)\times \phi^{-1}(B))\eqd\xi(A\times B)$ for every measure preserving map $\phi:\R_+\to\R_+$.
\end{defn}

\begin{defn}[Multigraphex]\label{def:multigraphex}
\normalfont
A multigraphex is a triple $\cW = (W,S,I)$ such that $I\in \ell_1$, $S:\R_+  \mapsto \ell_1$ is a measurable function, and $W:\R_+^2\times\N_0\mapsto \R_+$ is a measurable function satisfying  $W(x,y,k) = W(y,x,k),$  $\sum_{k=0}^{\infty} W(x,y,k) = 1,$ for any $x,y\in\R_+$ and $k\in\N_0$.
We will assume throughout that, $ \min\{ \sum_{k \geq 1}S(\cdot, k), 1\}$ is integrable. Further, setting $\mu_W(\cdot) = \int (1- W( \cdot ,y,0) )\dif y$, we assume that
\begin{enumerate}[(a)]
\item $\Lambda(\{x: \mu_W(x) = \infty \})=0$ and $\Lambda(\{x: \mu_W(x) >1 \})<\infty$,
\item $\int (1- W(x,y,0))\ind{\mu_W(x) \leq 1 }\ind{\mu_W(y) \leq 1 } \dif y \dif x <\infty$,
\item $\int  (1- W(x,x,0))\dif x <\infty$.
\end{enumerate}
\end{defn}

\begin{defn}[Adjacency measure of a multigraphex] \label{defn:adj_multigraphex}
\normalfont
Given any multigraphex $\cW = (W,S,I)$,
define $\xi_{\sss \cW}$, the \emph{random adjacency measure
generated by $\cW$}  as follows:
\begin{align*}
\xi_{\sss \cW} &= \sum_{i\neq j} \zeta_{ij} \delta_{(\theta_i,\theta_j)} + \sum_{i} \zeta_{ii} \delta_{(\theta_i,\theta_i)}+ \sum_{j,k} g(\theta_j, \chi_{jk})\big(\delta_{(\theta_j,\sigma_{jk})}+\delta_{(\sigma_{jk},\theta_j)}\big)
\\
&\hspace{1cm}+ \sum_k  h(\eta_k'') \big(\delta_{(\eta_k,\eta_k')}+\delta_{(\eta_k',\eta_k)}\big),\\
\zeta_{ij} &= r, \quad \text{if } \sum_{l=0}^{r-1}W(v_i,v_j,l) \leq U_{\{i,j\}} \leq  \sum_{l=0}^{r}W(v_i,v_j, l),\\
g(\theta_j, \chi_{jk}) &= r, \quad \text{if } \sum_{l=0}^{r-1} S(v_j, l) \leq \chi_{jk} \leq \sum_{l=0}^{r} S(v_j, l), \\
h(\eta_k '') &= r, \quad\text{if } \sum_{l=0}^{r-1} I(l) \leq \eta_k'' \leq \sum_{l=0}^{r} I(l).
\end{align*}
where $(U_{\{i,j\}})_{i,j\geq 1}$ is a collection of independent uniform[0,1] random variables, $\{(\theta_j,v_j)\}_{j\geq 1}$, $\{(\chi_{jk},\sigma_{jk})\}_{k\geq 1}$ for all $j\geq 1$ are unit rate Poisson point processes on $\R_+^2$, and $(\eta_k,\eta_k',\eta_k'')_{k\geq 1}$ is a unit rate Poisson point processes on $\R_+^3$, where all the above Poisson point processes are independent of each other and $(U_{\{i,j\}})_{i,j\geq 1}$.
\end{defn}

\begin{proposition}\label{prop:rand_adj_characterization}
Every  random adjacency measure is the adjacency measure corresponding to some (possibly random) multigraphex.
\end{proposition}

\begin{remark} \normalfont
The corresponding characterization for graphexes was stated  in \citet[Theorem 4.9]{VR15}. Their proof is based on the Kallenberg representation theorem for exchangeable random measures on $\R_+^2$ \cite{Kal90}, and the characterization for exchangeable random measures to be locally finite a.s. \cite[Prop 9.25]{K2005}. The missing condition in the  local finiteness criterion \cite[Prop 9.25]{K2005}, however, necessitates a slight modification to their proof. Armed with Theorem~\ref{thm:local_finiteness}, we generalize the result of \citet[Theorem 4.9]{VR15} to multigraphexes, and provide a proof sketch in the rest of this section.
\end{remark}

Armed with Corollary~\ref{cor:atomic_representation} and Theorem~\ref{thm:local_finiteness}, we furnish a proof of Proposition~\ref{prop:rand_adj_characterization}.

\begin{proof}[Proof of Proposition~\ref{prop:rand_adj_characterization}]
Given Corollary~\ref{cor:atomic_representation} and Theorem~\ref{thm:local_finiteness}, the proof is similar to that of Theorem 4.7, 4.9 in \cite{VR15}. Thus we sketch the proof, and refer the interested reader to \cite{VR15} for complete details.

Let $\xi$ be a random adjacency measure. Corollary~\ref{cor:atomic_representation} immediately implies that $\xi$ has a representation of the form \eqref{eq:atomic_rep}. Now, symmetry of $\xi$  enforces
$f(\cdot, x,y, \cdot) = f(\cdot, y, x, \cdot)$, $g = g'$ and $l = l'$. Further, $\xi \in \cN(\R^2_+)$, which specifies that $f$, $g$ and $l$ are actually $\N_0$ valued in this case.  Given these observations,
for any fixed $a$ and $k \in \N$, we define
\begin{align}
W(a,x,y,k) =& \lambda\{ z\in [0,1]: f(a, x,y, z)= k \} . \nonumber  \\
S(a, x, k)  =& \lambda \{ y \in \R_+ : g(a, x, y) = k \} \nonumber \\
I(a, k) =& \lambda\{ y \in \R_+ : l(a,y) =k \}. \nonumber
\end{align}
Finally, we set $W(a,x,y,0) = 1 - \sum_{k =1}^{\infty} W(a, x,y ,k)$. Note that $\xi$ is locally finite, and thus by Theorem~\ref{thm:local_finiteness} condition $(ii)$, $g_1 < \infty$ for Lebesgue almost all $x \in \R_+$, ensuring that $S(a,x.k)$ is well defined. For all other $x$, we define $S(a, x, k)$ arbitrarily in $\ell_1$.  Similarly, Theorem~\ref{thm:local_finiteness} condition $(i)$ ensures that $I$ is well defined almost surely. Also, we note that $f(\cdot, x, y , \cdot) = f(\cdot, y, x, \cdot)$ enforces the symmetry of $W$. Finally, it is easy to see that the constraints of  Theorem~\ref{thm:local_finiteness} translate directly to the integrability conditions imposed on multigraphexes in Definition~\ref{def:multigraphex}.

Finally, it remains to establish that the random adjacency measure $\xi$ is completely specified, given the multigraphex $(I,S, W)$. This follows from the arguments delineated in the proof of \cite[Theorem 4.7]{VR15}, and is thus omitted.
\end{proof}

We end this section with a criterion for the multigraphex in Proposition~\ref{prop:rand_adj_characterization} to be non-random.
The proof can be carried out in an identical manner for point processes in $\cN(\R_+^2)$ as \cite[Lemma 3.4]{BCCV17} and thus omitted.
\begin{proposition}
Let $\Gamma \in \cN(\R_+^2)$ be a jointly exchangeable adjacency measure.
Then $\Gamma$ is extremal if and only if for all $0<r < r'<\infty$, $\Gamma([0, r)^2 \cap  \cdot)$ and $\Gamma([r, r')^2 \cap  \cdot)$ are independent.
\end{proposition}

\bibliographystyle{apa}
\bibliography{BCDS17}

\end{document}